\documentclass[11pt]
{amsart}
\usepackage{hyperref}
\usepackage{pdfpages}
\hypersetup{nesting=true,debug=true,naturalnames=true}
\usepackage{graphicx,amssymb,upref}

\newtheorem{theorem}{Theorem}

\newtheorem{definition}{Definition}
\title{Asymmetric five person Hat Game}
\author{Theo van Uem}  
\address{Amsterdam University of Applied Sciences, Amsterdam, The Netherlands.} 
\email{tjvanuem@gmail.com}  

\begin{document}
\hbadness=99999

\begin{abstract}
 The Hat Game (Ebert’s Hat Problem) got much attention in the beginning of this century; not in the last place by its connections to coding theory and computer science. All players guess simultaneously the color of their own hat observing only the hat colors of the other players. It is also allowed for each player to pass: no color is guessed. The team wins if at least one player guesses his or her hat color correct and none of the players has an incorrect guess. This paper studies Ebert’s hat problem with 5 players and two colors where the probabilities of the colors may be different (asymmetric case). Our goal is to maximize the probability of winning the game and to describe winning strategies. In this paper we use the notion of an adequate set. The construction of adequate sets is independent of underlying probabilities and we can use this fact in the analysis of the asymmetric case. Another point of interest is the fact that computational complexity using adequate sets is much less than using standard methods.
\end{abstract}
\maketitle
\section{Introduction}
Hat puzzles were formulated at least since Martin Gardner’s 1961 article~\cite{MG}. They have got an impulse by Todd Ebert in his Ph.D. thesis in 1998~\cite{TE}. Buhler~\cite{JB} stated: “It is remarkable that a purely recreational problem comes so close to the research frontier”. Also articles in The New York Times~\cite{SR}, Die Zeit~\cite{WB} and abcNews~\cite{JP} about this subject got broad attention. This paper studies  generalized Ebert’s hat problem (symmetric and asymmetric): five distinguishable players are randomly fitted with a colored hat (two colors available), where the probabilities of getting a specific color  may be different, but known to all the players. All players guess simultaneously the color of their own hat observing only the hat colors of the other four players. It is also allowed for each player to pass: no color is guessed. The team wins if at least one player guesses his or her hat color correctly and none of the players has an incorrect guess. Our goal is to maximize the probability of winning the game and to describe winning strategies.
The symmetric two color hat problem (equal probabilities for each color) with $N=2^k-1$ players is solved in~\cite{EMV}, using Hamming codes, and with $N=2^k$ players in~\cite{GC} using extended Hamming codes. 
Burke et al.~\cite{EB} try to solve the symmetric hat problem with $N=3,4,5,7$ players using genetic programming. Their conclusion: The $N$-prisoners puzzle (alternative names: Hat Problem, Hat Game) gives evolutionary computation and genetic programming a new challenge to overcome. 
Lenstra and Seroussi~\cite{HL} show that in the symmetric case of two hat colors, and for any value of $N$, playing strategies are equivalent to binary covering codes of radius one.
Combining the result of Lenstra and Seroussi with Tables for Bounds on Covering Codes~\cite{GK}, we get:\par
\begin{center}
\begin{tabular}{|c|c|c|c|c|c|c|c|c|}
\hline
$N$&2&3&4&5&6&7&8&9\\
\hline
$K(N,1)$&2&2&4&7&12&16&32&62\\
\hline
\end{tabular}
\par
\end{center}

$K(N,1)$ is smallest size of a binary covering code of radius 1. Maximum probability for Ebert's symmetric Hat Game is $1-\frac{K(N,1)}{2^N}$. Lower bound on $K(9,1)$ was found in 2001 by Östergård-Blass, the upper bound in 2005 by Östergård.
Krzywkowski~\cite{MK} describes applications of the hat problem and its variations, and their connections to different areas of science.  
Johnson~\cite{BJ}  ends his presentation with an open problem:
If the hat colors are not equally likely, how will the optimal strategy be affected?
We will answer this question for five persons and two colors and our method gives also interesting results in the symmetric case.
In section 2 we define an adequate set. 
In section 3 we obtain results for the asymmetric five player, two color Hat Game.
In section 4 we do the same  for  the symmetric Hat Game. 
In all situations all players know the underlying probabilities of each player.

\section{Adequate sets}
In this section we have $N$ players and $q$ colors.
The $N$ persons in our game are distinguishable, so we can label them from 1 to
$N$. We label the $q$ colors $0,1,..,q-1.$ The probabilities of the colors are
fixed and known to all players. The probability that color $i$ will be on a hat is
$p_i$ $(i\in\{0,1,..,q-1\},\ \ \sum_{i=0}^{q-1}p_i=1).$\\
Each possible configuration of the hats can be represented by an element of
$B=\{b_1b_2\dots b_N\vert b_i\in\left\{0,1,\dots,q-1\right\},\ i=1,2..,N\}$.
The S-code represents what the $N$ different players sees. Player $i$ sees q-ary code $b_1..b_{i-1}b_{i+1}..b_N$ with decimal value
$s_i=\sum_{k=1}^{i-1}b_k.q^{N-k-1}+\sum_{k=i+1}^Nb_k.q^{N-k}\ $, a value between 0 and $q^{N-1}-1.$ \\
Let S be the set of all S-codes:
$S=\{s_1s_2\dots s_N\vert{}s_i=\sum_{k=1}^{i-1}b_k.q^{N-k-1}+\sum_{k=i+1}^Nb_k.q^{N-k},b_i\in{}\{0,1,\dots,q-1\},\
i=1,2,\dots,N\ \}$.\\
Each player has to make a choice out of $q+1$ possibilities: 0='guess color 0', 
1='guess color 1', \ldots{}.,$\ q-1$ ='guess color $q-1$', $q$='pass'.
\\
We define a decision matrix $D=\left(a_{i,j}\right) \ $ where
$i\in{}\{1,2,..,N\}$(players); $j\in{}\{0,1,..,q^{N-1}-1\}$(S-code of a player);
$a_{i,j}\in{}\left\{0,1,..,q\right\}.$ \\
The meaning of $a_{i,j}$ is: player $i$ sees S-code $j$ and takes decision $a_{i,j}$ (guess a color or pass).
We observe the total probability (sum) of our guesses. 
For each $b_1b_2\dots b_N$ in B  with $n_i$ times color $i$   $(i=0,1,\dots,q-1,\
\sum_{i=0}^{q-1}n_i=N\ $) we have:\newline

CASE  $b_1b_2\dots b_N$  

 (S-code player $i$:
$s_i=\sum_{k=1}^{i-1}b_k.q^{N-k-1}+\sum_{k=i+1}^Nb_k.q^{N-k}$)\\
IF $a_{1{,s}_1}\in{}\{q,b_1\}$  AND $a_{2,s_2}\in{}\{q,b_2\}$  AND ... AND 
$a_{N,s_N}\in{}\{q,b_N\}$  AND \\ NOT
($a_{1,s_1}=a_{2{,s}_2}=\dots=a_{N,s_N}=q)$  THEN sum=sum+$p_0^{n_0}.p_1^{n_1}\dots\
p_{q-1}^{n_{q-1}}$.
\newline

Any choice of the $a_{i,j}$ in the decision matrix determines which CASES $b_1b_2\dots b_N$ have a
positive contribution to sum (we call it a GOOD CASE) and which CASES don't
contribute positive to sum (we call it a BAD CASE).
\begin{definition}
 Let $A \subset B$. $A$ is adequate to $B-A$ if for each q-ary element $x$ in $B-A$ there are $q-1$ elements in A which are equal to $x$ up to one fixed q-ary position.
 \end{definition}
 
\begin{theorem}
 BAD CASES are adequate to GOOD CASES.
\end{theorem}
\begin{proof}
  Any  GOOD CASE  has at least one $a_{i,j}$ not equal to $q$. Let this
specific $a_{i,j}$ have value $b_{i_0}.$ Then our GOOD CASE generates $q-1\ $BAD CASES
by only changing  the value $b_{i_0}$ in any value of $0,1,..,q-1\ $
except $b_{i_0}$.  
\end{proof}
The definition of adequate set is the same idea as the concept of 
strong covering, introduced by Lenstra and Seroussi~\cite{HL}. The number of elements in an adequate set will be written as \textit{das} (dimension of adequate set).
Adequate sets are generated by an adequate set generator (ASG) We have implemented an ASG in VBA/Excel (see Appendix \ref{appendix:one}). \\
Given an adequate set, we obtain a decision matrix $ D=\left(a_{i,j}\right)$ by the following procedure.\\ 
Procedure DMG (Decision Matrix Generator):\\
Begin Procedure\\
For each element in the adequate set:
\begin{itemize}
\item Determine the q-ary representation $b_1b_2\dots b_N$
 \item Calculate S-codes $s_i=\sum_{k=1}^{i-1}b_k.q^{N-k-1}+\sum_{k=i+1}^Nb_k.q^{N-k}$($i$=1..N)
\item For each player $i$: fill decision matrix with  $a_{{i,s}_i}=b_i$  ($i$=1,..,N), where each cell may contain several values.
\end{itemize}
Matrix $D$ is filled with BAD COLORS. We can extract the GOOD COLORS by considering all $a_{i,j}$ with $q-1$ BAD COLORS and then choose the only missing color. In all situations with less than $q-1$ BAD COLORS we pass.
When there is an $a_{i,j}$ with $q$ BAD COLORS all colors are bad, so the first
option is to pass. But when we  choose any color, we get a situation with
$q-1$ BAD COLORS.  So in case of $q$ BAD COLORS we are free to choose any color or
pass.
The code for pass is $q$, but in our decision matrices we prefer a blank, which
supports readability.
The code for `any color or pass will do' is defined  $q+1$, but in our decision
matrices we prefer a $\star$. \\
End Procedure.\\
Decision matrices are generated by an Decision Matrix Generator (DMG).We have implemented a DMG in VBA/Excel (see Appendix \ref{appendix:two}). 
\section{Asymmetric five person Hat Game}\label{threethree}
\begin{theorem}
In asymmetric five person (two color) hat game we have maximal probability $\Psi{}\left(5,p\right)$ of winning the game:
\begin{equation} \label{100}
\Psi{}\left(5,p\right)=\left\{\begin{array}{l}1-p+2p^2-2p^3+p^5\ \ \ \ \ \ \ \ \
\ \ \ \ \ \ \ \ \  (0\leq{}p\leq{}\sqrt{2}-1) \\
5p-10p^2+6p^3+p^4-p^5\ \ \ \ \ \ \ \ \ \ \ \ \ \ (\sqrt{2}-1\leq{}p\leq{}0.5) \\
1-2p+4p^2-4p^4+p^5\ \ \ \ \ \ \ \ \ \ \ \ \ \ \ \ \ (0.5\leq{}p\leq{}2-\sqrt{2})
\\
1-2p+6p^2-8p^3+5p^4-p^5\ \ \ \ \ \ \ \ \ \
(2-\sqrt{2}\leq{}p\leq{}1)\end{array}\right.
\end{equation}
with  optimal decision matrices:\\
CASE $0.5<p \leq 2-\sqrt {2}$:\\
\includegraphics[scale=0.7]{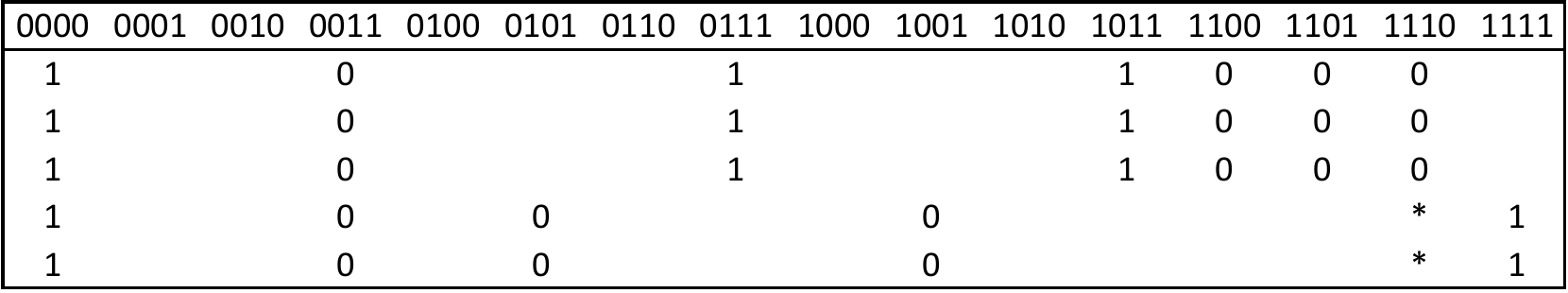} \label{101}\\
CASE $2-\sqrt{2} \leq p<1$:\\
\includegraphics[scale=0.7]{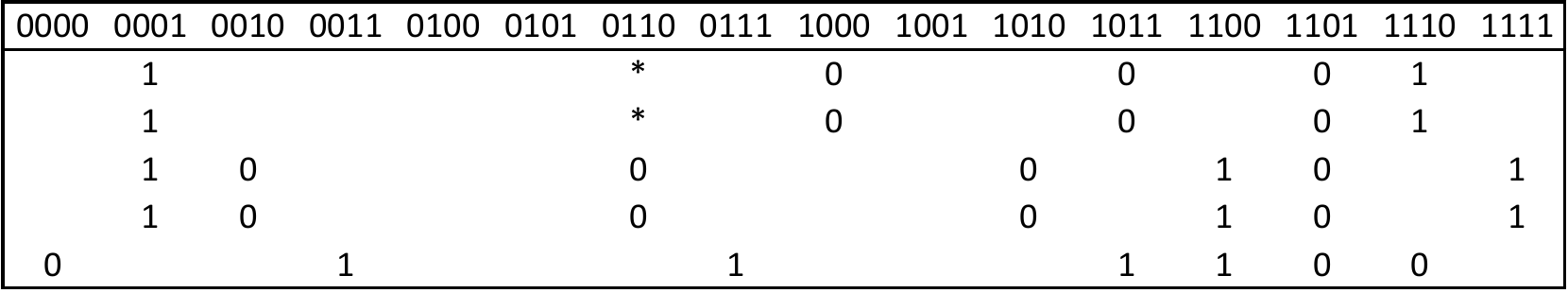} \label{102}\\
{\raggedright
($\star$ means: any color or pass; stars are independent)
}
\end{theorem}
\begin{proof}
We run the application ASG with parameters $N$=5, $p$=0.55 and \textit{das}=7.  This yields 320 adequate sets. When we sort these sets by probability, we get 12 different classes (see Appendix \ref{appendix:four} where we show the first element of each class).
The adequate set  and the number of zero's in each element in an adequate set
are independent of $p$. We note the following structure in Appendix \ref{appendix:four}:
\begin{center}
    \begin{tabular}{|l|l|}
    \hline
    012345  & {probability} \\
    \hline
    024001 &$2pq^4+4p^2q^3+p^5$  \\
      022210 & $2pq^4+2p^2q^3+2p^3q^2+p^4q$ \\
   {111310} &$q^5+pq^4+p^2q^3+3p^3q^2+p^4q$  \\
    013210 & $pq^4+3p^2q^3+2p^3q^2+p^4q$ \\
   {102310} &$q^5+2p^2q^3+3p^3q^2+p^4q$  \\
    012310 &$pq^4+2p^2q^3+3p^3q^2+p^4q$  \\
       {120130} & $q^5+2pq^4+p^3q^2+3p^4q$ \\
       013201 &$pq^4+3p^2q^3+2p^3q^2+p^5$  \\
    {012220} & $pq^4+2p^2q^3+2p^3q^2+2p^4q$ \\
     031021 & $3pq^4+p^2q^3+2p^4q+p^5$ \\
        013111 & $pq^4+3p^2q^3+p^3q^2+p^4q+p^5$ \\
       {100420} &$q^5+4p^3q^2+2p^4q$  \\
       \hline
    \end{tabular}%
\end{center}
\begin{figure}
\includegraphics[width=\textwidth]{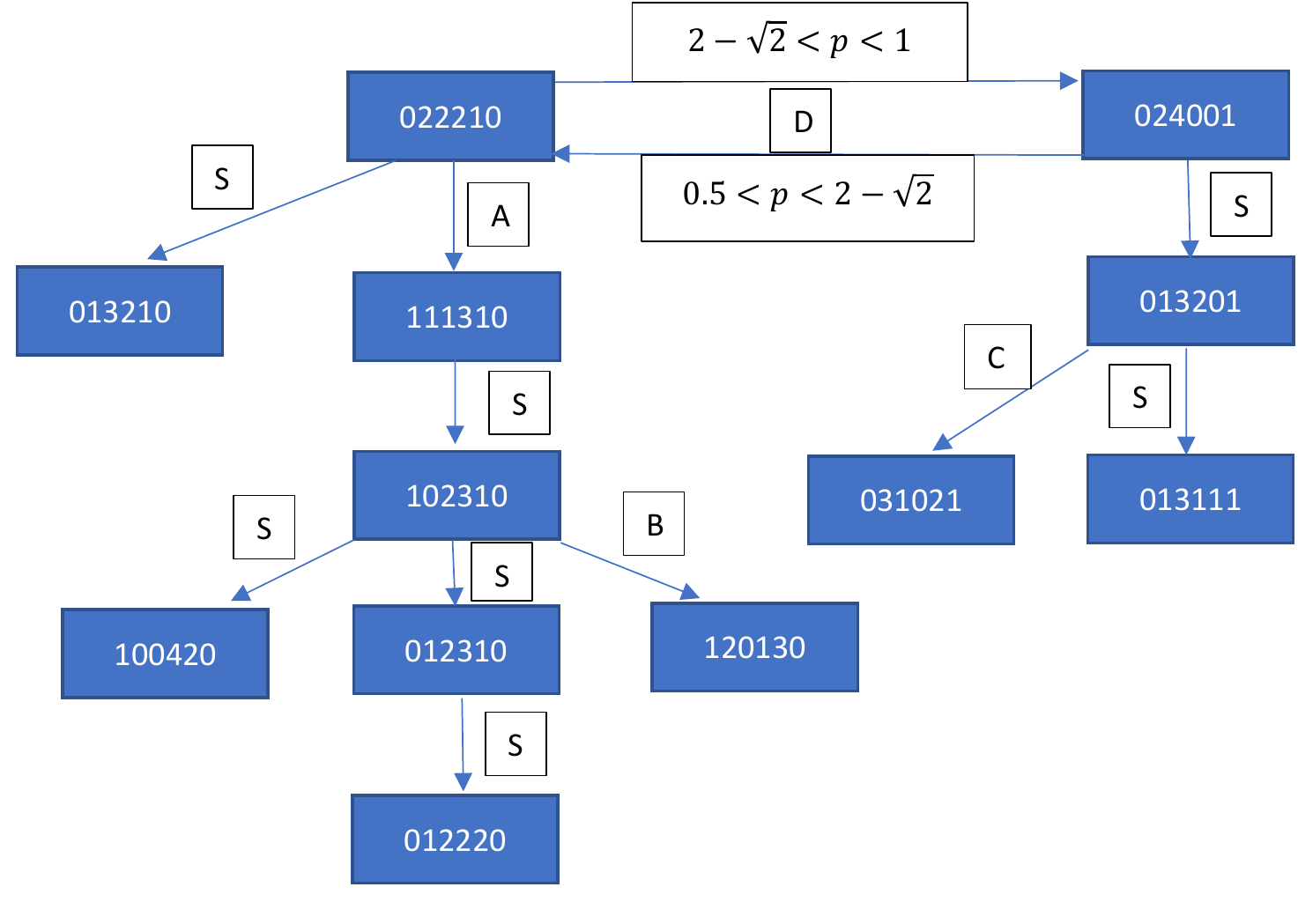}
\caption{Dominance relations}
\label{SSS}
\end{figure}
{\raggedright
Figure \ref{SSS} shows the dominance relations between this 12 different
classes.\\
}
{\raggedright
A class is dominant over another class if the probability of that class is less
than the probability of the other class.
} \\
{\raggedright
An arrow with a S means that one or more shifts of bits to the left (of the
dominated one) results to the pattern of  the dominant one (shifts to the left
are 'cheaper', because of $p$$>$$q$).
}
More formally, for fixed $das$ we have: the pattern b(0) b(1) b(2) b(3) b(4) b(5) is dominant over pattern a(0) a(1) a(2) a(3) a(4) a(5) when in each position there is enough compensation:\\ \\
b(0)$\geq$ a(0)
\newline
b(0)+b(1) $\geq$ a(0)+a(1) 
\newline
...
\newline
...
\newline
b(0)+b(1)+b(2)+b(3)+b(4)+b(5) $\geq$ a(0)+a(1)+a(2)+a(3)+a(4)+a(5)
\newline

The last inequality may be omitted (we have $\textit{das} \geq \textit{das}$). \\
The dominance relation is transitive.

{\raggedright
The arrows A,B,C and D needs some attention.\\
 We use  $\succ{}$  as symbol for
dominance.\\
}
{\raggedright
A:   022210$\ \succ{}$ 111310
}\\
If $p>q>0$ then $-q^5+pq^4+ p^2q^3-p^3q^2=
  -q^5+pq^4+{\ p}^2q^3-p^3q^2=-q^2{(p-q)}^2<0$
\\
{\raggedright
B and C:   102310 $\succ{}$ 120310    and    013201 $\succ{}$ 031021 \\
}
{\raggedright
if $p$$>$$q$ $>$0 then 
  $-2pq^4+{2p}^2q^3+{\ 2p}^3q^2-{\ 2\ p}^4q=-2pq{\left(p-q\right)}^3<0\ 
$\\
}
{\raggedright
D:  If \ $2-\sqrt{2}<p<1\ $ then 022210 $\succ{}$ 024001 and
if  \ $0.5<p<2-\sqrt{2}$ then   024001 $\succ{}$ 022210 \\
}
{\raggedright
  ${-2p}^2q^3+{ 2p}^3q^2+{ p}^4q-{ p}^5=p^2\left({
2p}^3-9p^2+8p-2\right)=p^2\left(2p-1\right)\left(p^2-4p+2\right)=p^2\left(2p-1\right)(p-2-\sqrt{2})(p-2+\sqrt{2})$
}\\
\newline
{\raggedright
Let $\Psi{}\left(N,p\right)$ be the maximum probability of correct guessing.\\
}
{\raggedright
If $2-\sqrt{2}<p<1\ $ then 022210 $\succ{}$ 024001 and
}
\[\Psi{}\left(5,p\right)=
1-\left(2pq^4+{2p}^2q^3+{2p}^3q^2+{p}^4q\
\right)=1-2p+{6p}^2-8p^3+{5p}^4-{p}^5
\]
{\raggedright
If $0.5<p<2-\sqrt{2}$ then   024001 $\succ{}$ 022210 and
}
\[
\Psi{}\left(5,p\right)=1-\left(2pq^4+{\ 4p}^2q^3+\
p^5\right)=1-2p+{4p}^2-{4p}^4+p^5
\]
{\raggedright
By changing the roles of $p$ and $q$ we get: \eqref{100}
} \\
{\raggedright
Graph of $\Psi{}\left(5,p\right)$:
}
\begin{center}
 \includegraphics[scale=0.5]{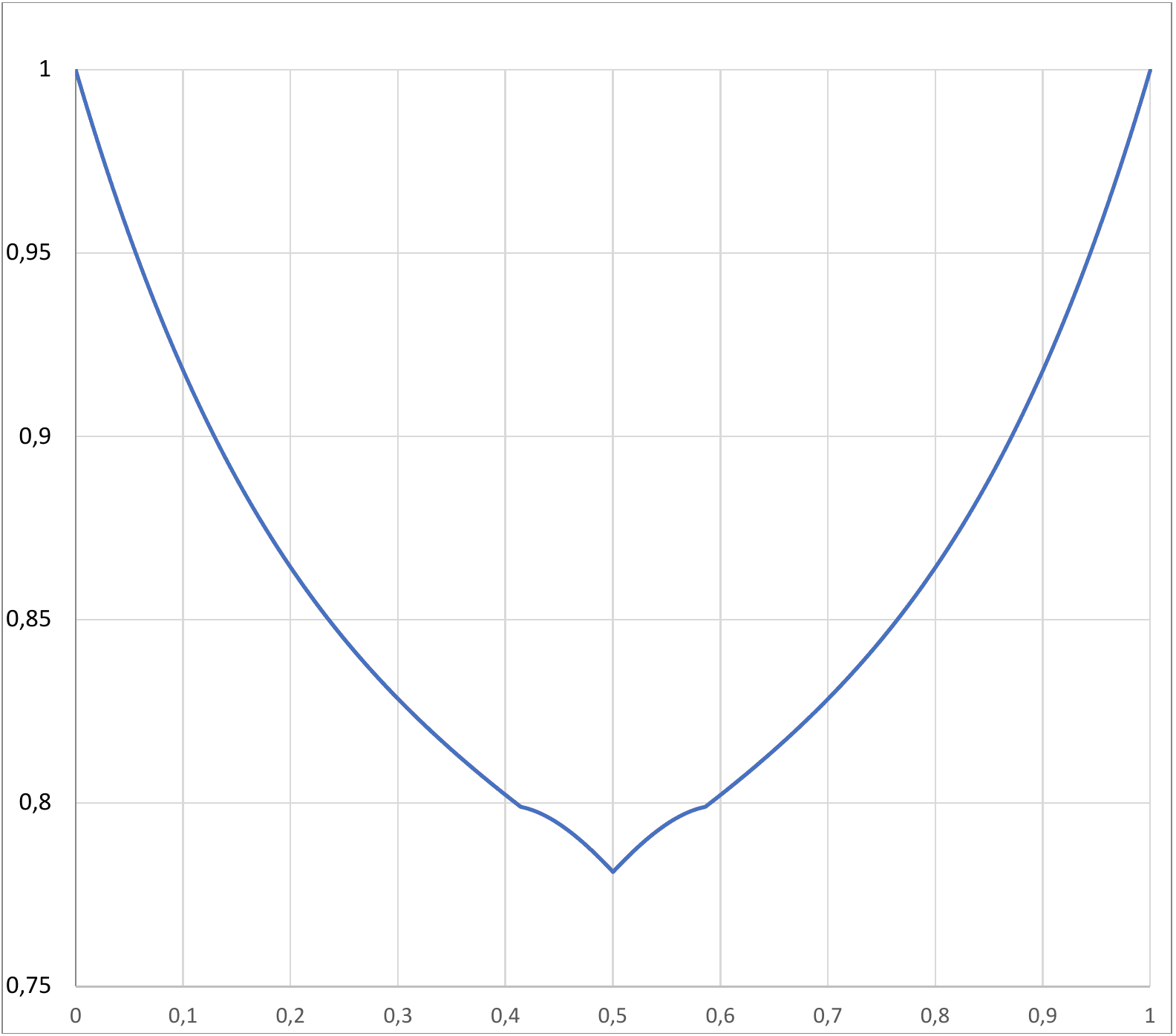}
  \end{center}
 {\raggedright
We remark: minimum is at $(\frac{1}{2},\frac{25}{32})$ and $\Psi{}\left(5,p\right)$ is not
differentiable at $\frac{1}{2}$, $\sqrt{2}-1\ $and $2-\sqrt{2}$.
}
{\raggedright
When $N$=5 we have 320 adequate sets.
}
{\raggedright
Using the program ASG \textit{p}=0.55 \textit{das}=7   we get after sorting on sum:
}
{\raggedright
when $0.5<p<2-\sqrt{2}$ we have 10 optimal adequate sets.
}
{\raggedright
Using the program ASG  \textit{p}=0.9 \textit{das}=7   we get after sorting on sum:
when  $2-\sqrt{2}<p<1\ $ we have 30 optimal adequate sets. \\
}
{\raggedright
When  $\textit{p}=2-\sqrt{2}$  : 40 optimal adequate sets (the union of the two foregoing
adequate sets).
}\\
{\raggedright
Using procedure DMG  , we give the first element in each case: \\
}
{\raggedright
CASE $0.5<p<2-\sqrt{2}$ : adequate set \{0,7,11,19,28,29,30\};\\
Use: DMG  \textit{p}=0.55 to obtain the decision matrices.
}
{\raggedright
Note: players 1,2,3 have same strategy; players 4 and 5 also.
}
\newline
{\raggedright
CASE  $2-\sqrt{2}<p<1\ $ : adequate set \{1,6,14,22,24,27,29\};\\
Use DMG   to obtain the decision matrices.
}
\\
{\raggedright
Note: players 1,2  have same strategy, players 3,4 also; player 5 has her own
strategy.
}
\newline
{\raggedright
CASE  $\textit{p}=2-\sqrt{2}$: we get the union of 10 optimal sets in case of
$0.5<p<2-\sqrt{2}$  and 30 optimal sets in case of $2-\sqrt{2}<p<1\ $.
}\\ \\
We call two solutions isomorphic when we can transform one solution to the other by renumbering the players. Observing all 40 decision matrices and making use of the positions (players) of the stars in these matrices, we get the result in Appendix \ref{appendix:four2}, where the column STARS gives the positions of the  two stars and the column CYCLES gives the renumbering to obtain the same solution. This can be verified by writing each adequate set in the binary representation and then applying the cycles.
Appendix \ref{appendix:four2} shows that there are only two non-isomorphic solutions (these two solutions can't be isomorphic for isomorphic sets have always the same probability).
\newline \\
{\raggedright
The last point here is to convince ourselves  that any adequate set with
\textit{das}$>$7 doesn't yield better solutions.
\\ 
Let $b=das(b())$, $a=das(a())$.
When $e=a-b \geq 0$ then we delete the 'cheapest' $a-b$ elements from $a()$, which gives a 'cheaper' situation. We get: 
\newline
\newline
b(0)$\geq$ a(0)-e
\newline
b(0)+b(1) $\geq$ a(0)+a(1)-e 
\newline
...
\newline
...
\newline
b(0)+b(1)+b(2)+b(3)+b(4)+b(5) $\geq$ a(0)+a(1)+a(2)+a(3)+a(4)+a(5)-e
\newline \\
These dominance relations are implemented as procedure $dom()$ in ASG. \\
Run the program ASG   with \textit{das}=8,9,\ldots{}32.
} where we have 5 calls to  $dom()$ (see figure \ref{SSS}; we omitted the S patterns, for they are detected by the procedure). No (dominant) adequate sets are found. \\
Conclusion: $\textit{das}=7$ is optimal and $\{(022210),(024001)\}$ dominates all adequate sets.
\end{proof}

\subsection{Computational complexity}

{\raggedright
We consider the number of strategies to be examined to solve the hat problem
with $N$ players and two colors. Each of the $N$ players has $2^{N-1}$ possible
situations to observe and in each situation there are three possible guesses:
white, black or pass. So we have ${(3^{2^{N-1}})}^N$ possible strategies.
Krzywkowski [14] shows that is suffices to examine ${(3^{2^{N-1}-2})}^N$
strategies.
}

{\raggedright
The adequate set method has to deal where \{$i_1,i_2,..,i_{\textit{das}}\}$ with
$0\leq{}i_1<i_2<..<i_{\textit{das}}\leq{}2^N-1$.
}

{\raggedright
The number of strategies for fixed \textit{das} is the number of subsets of dimension \textit{das} of
\{0,1,\ldots{},$\ 2^N$ -1\}: $\left(\begin{array}{l}2^N \\
\textit{das}\end{array}\right)$.
}
But we have to test all possible values of \emph{\textit{das}}. So the correct expression is: $\sum_{\emph{\textit{das}}}
\left(\begin{array}{l}2^N \\
\textit{das}\end{array}\right)=2^{(2^N)}$.
{\raggedright
To get an idea of the power of the adequate set method, we compare the number of
strategies (brute force, Krzywkowski and adequate set method):
}
\begin{center}
\renewcommand{\arraystretch}{1.5}
    \begin{tabular}{|c|c|c|c|}
    \hline
    
   {N}  &   ${(3^{2^{N-1}})}^N$    &   ${(3^{2^{N-1}-2})}^N$   &$2^{(2^N)}$  \\
     
   \hline
       5         & 1.50E+38 & 2.50E+33 &4.29E+09 \\
    
    \hline
    \end{tabular}%
 \end{center}

\section{Symmetric two color five person Hat Game}
In this section with five persons the two colors have the same probability: $p=q=\frac{1}{2}$.
\begin{theorem}
The maximal probability for symmetric five person two color Hat Game is $\frac{25}{32}$. All 12 non-isomorphic optimal decision matrices are given in Appendices \ref{appendix:four3} and \ref{appendix:four4}.
\end{theorem}
\begin{proof}
Running ASG  with any value of $p$ gives 320 adequate sets. When $\textit{p}=0.5$ all these sets are optimal (with probability $\frac{25}{32}$). We split the 320 sets in 12 probability classes (see section \ref{threethree}). In Appendices \ref{appendix:four3} and \ref{appendix:four4} we give the first element of each class, the decision matrix and the number of isomorphic elements in each class, where isomorphic sets can be detected in the same way (STARS, CYCLES) as in section \ref{threethree}. All 320 decision matrices are generated by  procedure DMG  .\\ The 12 probability classes generates 12 non-isomorphic decision matrices.
\end{proof}

\appendix
\clearpage



\section*{Supplementary material}

## Appendix A. Adequate Set Generator

```
Dim k1, k2, k3, k4, k5, k6 As Integer
Public dominant As Integer
Private Declare Function SetThreadExecutionState _
    Lib "Kernel32.dll" _
        (ByVal esFlags As Long) _
    As Long
Sub dominant_adequate_set_generator()
Dim d() As Integer
Dim i() As Integer
Dim check() As Integer
Dim a() As Integer
Dim c() As Integer
Dim j() As Integer
p = 0.55
q = 0.45
Count = 0   'number of already found adequate sets
n = 5       'five players
H = 2 ^ n - 1    '2^n elements in B
das = 7     'dimension of adequate set
ReDim a(0 To 5)
ReDim c(0 To H, 1 To n) As Integer
ReDim j(0 To H, 1 To n) As Integer
ReDim d(1 To n) As Integer
ReDim i(1 To das) As Integer
ReDim check(0 To H) As Integer
KeepPowerOn
'' for each number from 0 to H:
'' first calculate binary digits
'' 00000 ..... 11111 and put it in matrix c:
For k = 0 To H
    g = k
```

```
    For Z = 1 To n
       c(k, Z) = g Mod 2
       g = g \ 2
    Next Z
Next k
x = 0 ' x: row in Excel where result is displayed
' adequate set: {i_1,i_2..,i_das}
'brute force for all potential adequate sets:
For i1 = 0 To H + 1 - das
i(1) = i1
For i2 = i1 + 1 To H + 2 - das
i(2) = i2
For i3 = i2 + 1 To H + 3 - das
i(3) = i3
For i4 = i3 + 1 To H + 4 - das
i(4) = i4
For i5 = i4 + 1 To H + 5 - das
i(5) = i5
For i6 = i5 + 1 To H + 6 - das
i(6) = i6
For i7 = i6 + 1 To H + 7 - das
i(7) = i7
''binary digits for adequate set:
For k = 1 To das
   g = i(k)
   For Z = 1 To n
      j(k, Z) = g Mod 2
      g = g \ 2
   Next Z
Next k
'' check on adequate set property;
```

```
'' each element of B has distance 0 or 1 to A

For k = 0 To H

 check(k) = 0

  For m = 1 To das

  '' distance<2

  If Abs(c(k, 1) - j(m, 1)) + Abs(c(k, 2) - j(m, 2)) + Abs(c(k, 3) - j(m, 3)) + Abs(c(k, 4) - j(m, 4)) + Abs(c(k, 5) - j(m, 5)) < 2 Then check(k) = 1

   Next m

Next k

State = 1    'potential adequate set

For k = 0 To H

 State = State * check(k)

Next k

Sum = 0

Cells(1, das + 2) = Count

If State = 1 Then  ''State=1 means adequate set

  For t = 0 To 5

    a(t) = 0

' a(t) counts the number of p^t*q^(5-t) in adequate set

    Next t

  For k = 1 To das

    g = i(k)

    For Z = 1 To n

      d(Z) = g Mod 2

      g = (g - d(Z)) / 2

    Next Z

    count1 = 0

    count2 = 0

    For t = 1 To n

      If d(t) = 0 Then count1 = count1 + 1

      If d(t) = 1 Then count2 = count2 + 1
```

```
    Next t
    Sum = Sum + _
    p ^ count1 * q ^ count2
    a(count1) = a(count1) + 1
  Next k
End If
k1 = a(0): k2 = a(1): k3 = a(2)
k4 = a(3): k5 = a(4)
k6 = a(5)
dominant = 1
'a(.) is potential dominant (code 1);
'if dominated then switch: dominant=0;
'Call dom(0, 2, 2, 2, 1, 0)
'Call dom(0, 2, 4, 0, 0, 1)
'Call dom(1, 1, 1, 3, 1, 0)
'Call dom(1, 2, 0, 1, 3, 0)
'Call dom(0, 3, 1, 0, 2, 1)
If State = 1 And dominant = 1 Then
' dominant adequate set
x = x + 1
Count = Count + 1
'next row in Excel sheet:
'show elements of adequate set and additional information
For k = 1 To das
Cells(x, k) = i(k)
Next k
Cells(x, 3 + das) = Sum
For b = 0 To 5
    Cells(x, das + 4 + b) = a(b)
  Next b
End If
```

## Appendix B. Decision Matrix generator

```
Sub decision_matrix_generator_n5_q2()
Dim d() As Integer
Dim c() As Integer
Dim i() As Integer
Dim j() As Integer
Dim check() As Integer
Dim a() As Integer
Dim tot() As Integer
p = 0.55
q = 1 - p
n = 5
H = 2 ^ n - 1
v = 2 ^ (n - 1) - 1
das = 7
ReDim d(1 To n) As Integer
ReDim c(0 To H, 1 To n) As Integer
ReDim i(1 To das) As Integer
ReDim j(1 To das, 1 To n) As Integer
ReDim check(0 To H) As Integer
ReDim a(1 To n, 0 To v) As Integer
ReDim tot(1 To n, 0 To v) As Integer
'' for each number from 0 to H:
'' first calculate binary digits
'' 00000 ..... 11111 and put it in matrix c:
For k = 0 To H
   g = k
   For Z = 1 To n
      c(k, Z) = g Mod 2
      g = g \ 2
   Next Z
Next k
```

```
x = -5 ' x: row in Excel where result is displayed
' adequate set: {i_1,i_2..,i_das}
For i1 = 0 To H - das + 1
i(1) = i1
'' VBA-EXCEL can't handle with an array in for to next
For i2 = i1 + 1 To H - das + 2
i(2) = i2
For i3 = i2 + 1 To H - das + 3
i(3) = i3
For i4 = i3 + 1 To H - das + 4
i(4) = i4
For i5 = i4 + 1 To H - das + 5
i(5) = i5
For i6 = i5 + 1 To H - das + 6
i(6) = i6
For i7 = i6 + 1 To H - das + 7
i(7) = i7
''binary digits for adequate set:
For k = 1 To das
   g = i(k)
   For Z = 1 To n
      j(k, Z) = g Mod 2
      g = g \ 2
   Next Z
Next k
'' check on adequate set property;
'' each element of B has distance 0 or 1 to A
For k = 0 To H
 check(k) = 0
   For m = 1 To das
   '' distance<2
```

```
    If Abs(c(k, 1) - j(m, 1)) + Abs(c(k, 2) - j(m, 2)) + Abs(c(k, 3) - j(m, 3)) + Abs(c(k, 4) - j(m, 4)) + Abs(c(k, 5) - j(m, 5)) < 2 Then check(k) = 1
    Next m
  Next k
  State = 1
  For k = 0 To H
    State = State * check(k)
  Next k
  Sum = 0
  If State = 1 Then
    For k = 1 To das
      g = i(k)
      For Z = 1 To n
        d(Z) = g Mod 2
        g = g \ 2
      Next Z
      Count = 0 '' count of zero's in element of adequate set
      For t = 1 To n
        If d(t) = 0 Then Count = Count + 1
      Next t
      term = p ^ Count * q ^ (n - Count)
      Sum = Sum + term
    Next k
  End If
  If State = 1 Then
  x = x + n + 2
  Cells(x, 1) = 1 + x \ 7
  For k = 1 To das
    Cells(x, k + 2) = i(k)
  Next k
  For s = 1 To n
```



\section{Q=2, N=5,  12 classes}
\label{appendix:four}
\includegraphics[scale=0.7]{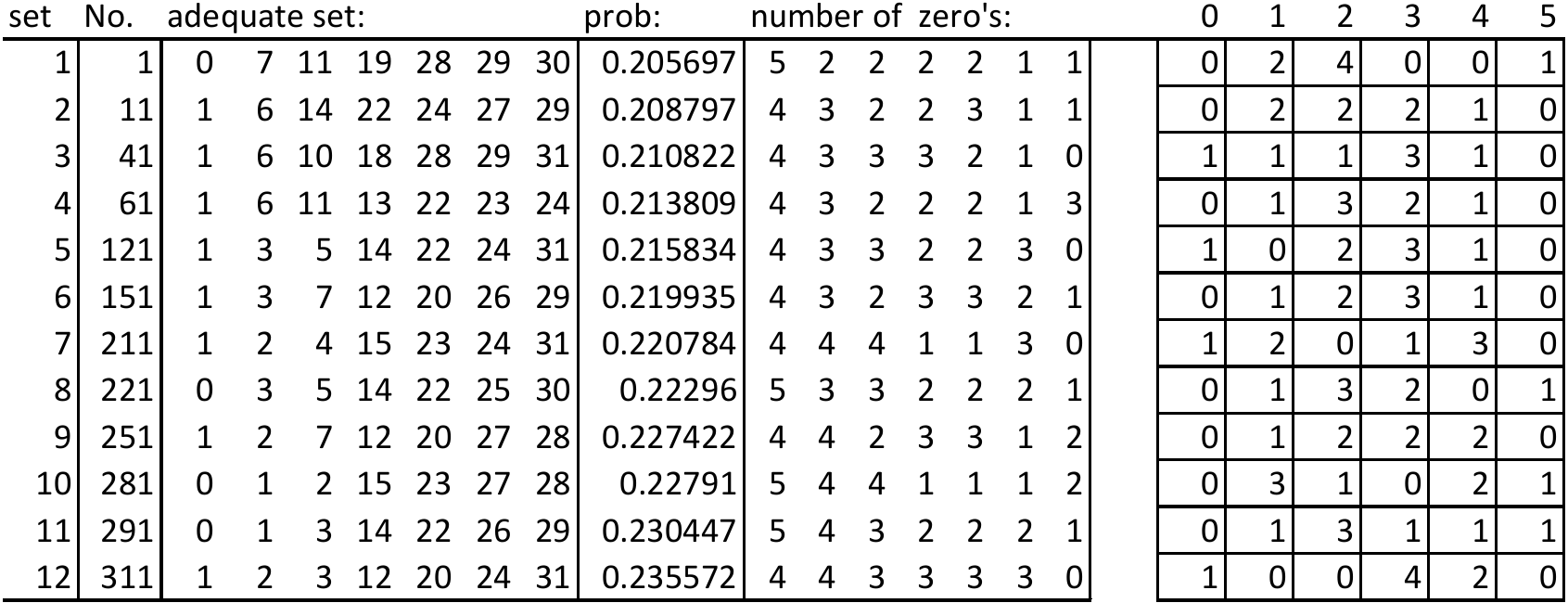}

 \section{Q=2, N=5, two isomorphic sets}
\label{appendix:four2}
\includegraphics[scale=0.7]{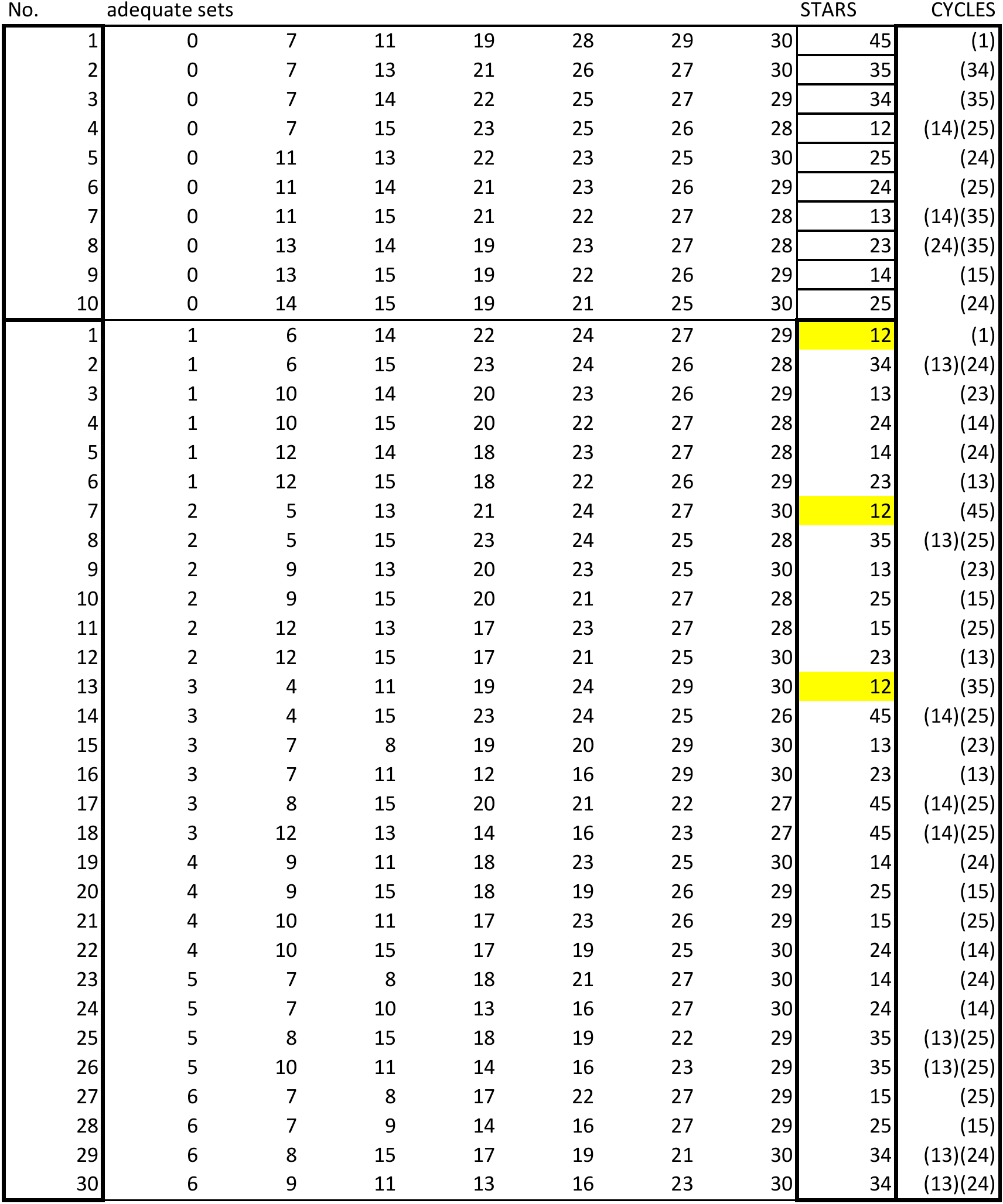}

\section{Q=2, N=5, decision matrices 12 classes, PART I}
\label{appendix:four3} 
 \includegraphics[scale=1.3]{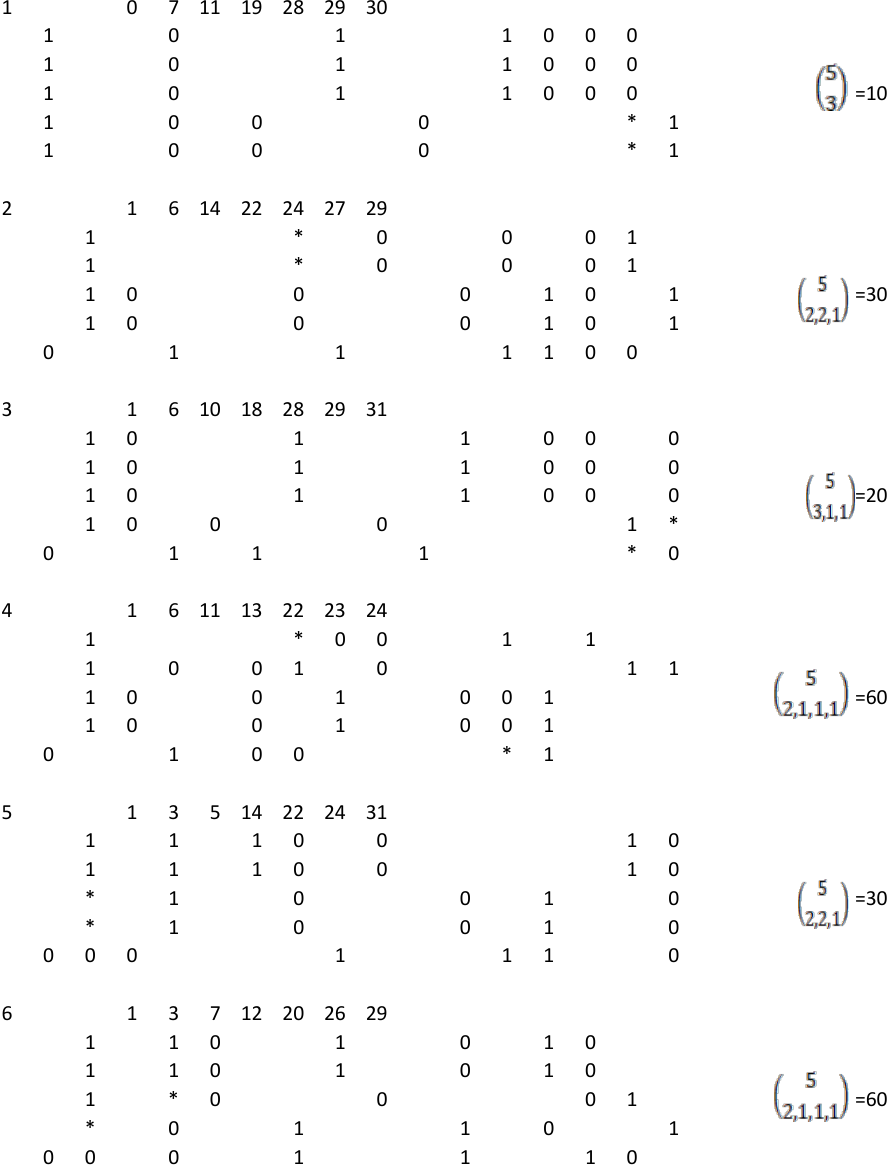}

\section{Q=2, N=5, decision matrices 12 classes, PART II}
\label{appendix:four4} 
 \includegraphics[scale=0.7]{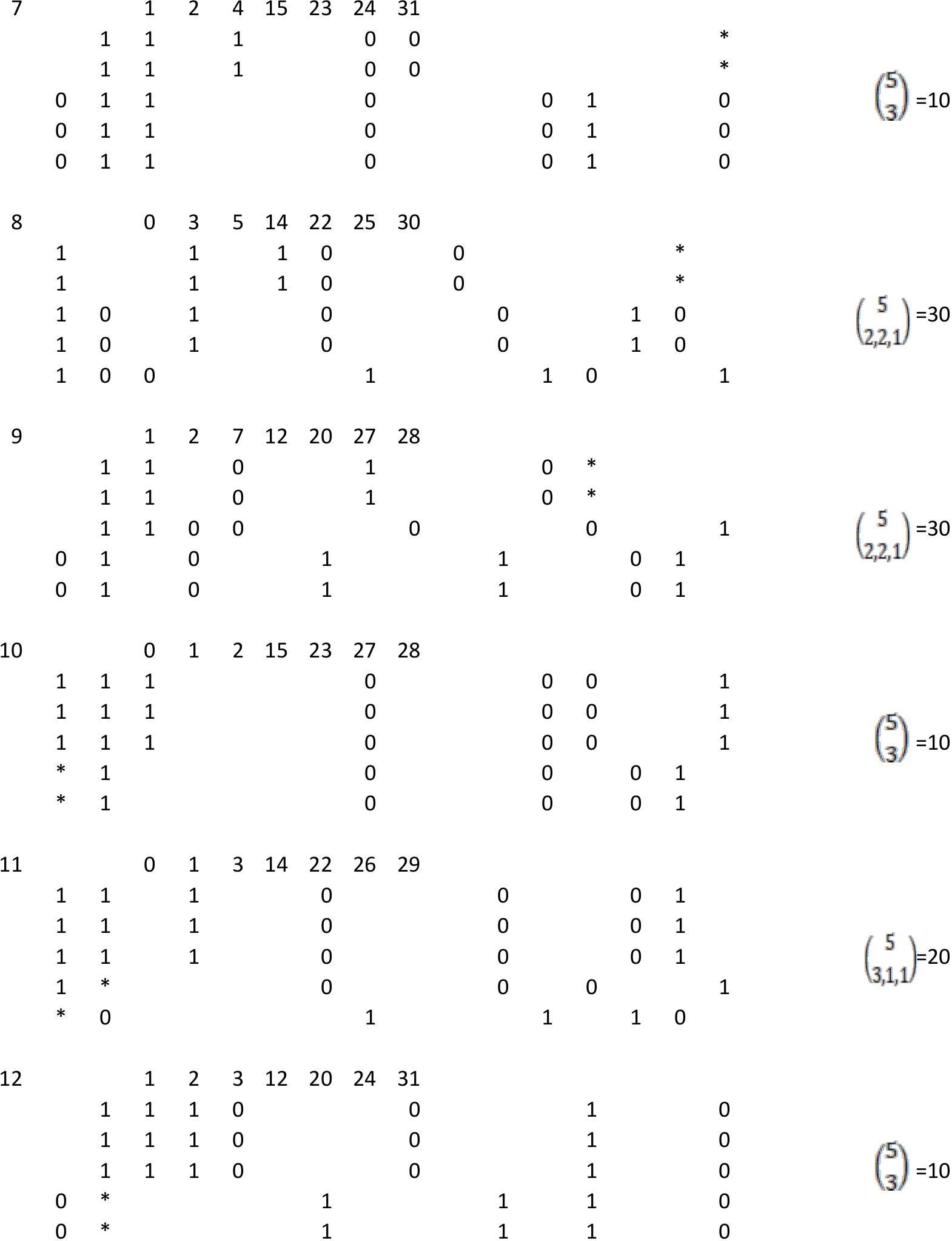}

\clearpage

\end{document}